\documentclass[12pt]{amsart}
\usepackage{a4wide}
\usepackage{amssymb}
\usepackage{amsthm}
\usepackage{amsmath}
\usepackage{amscd}
\usepackage{verbatim,url}
\usepackage{color,cancel}
\usepackage[mathscr]{eucal}
\textheight8.8in \textwidth6.55in \numberwithin{equation}{section}

\newtheorem{theorem}{Theorem}
\newtheorem{lemma}[theorem]{Lemma}
\newtheorem{corollary}[theorem]{Corollary}

\theoremstyle{remark}

\newtheorem*{remarks}{Remarks}

\numberwithin{theorem}{section} \numberwithin{equation}{section}
\numberwithin{figure}{section}

\newcommand{\R}{\mathbb{R}}

\newcommand{\Q}{\mathbb{Q}}

\newcommand{\Z}{\mathbb{Z}}
\newcommand{\N}{\mathbb{N}}

\renewcommand{\(}{\left(}
\renewcommand{\)}{\right)}

\begin{document}
\title{On the asymptotic behavior of unimodal rank generating functions}

\author[K. Bringmann]{Kathrin Bringmann} 
\address{Mathematical Institute\\University of
Cologne\\ Weyertal 86-90 \\ 50931 Cologne \\Germany}
\email{kbringma@math.uni-koeln.de}

\author[B. Kim]{Byungchan Kim}
\address{School of Liberal Arts \\ Seoul National University of Science and Technology \\ 232 Gongneung-ro, Nowongu, Seoul,139-743, Republic of  Korea}
\email{bkim4@seoultech.ac.kr}

\date{\today}
\thanks{The research of the first author was supported by the Alfried Krupp Prize for Young University Teachers of the Krupp foundation and the research leading to these results has received funding from the European Research Council under the European Union's Seventh Framework Programme (FP/2007-2013) / ERC Grant agreement n. 335220 - AQSER. The research of the second author was supported by the Basic Science Research Program through the National Research Foundation of Korea (NRF) funded by the Ministry of Education (NRF- 2013R1A1A2061326).} 
\subjclass[2000] {}
\keywords{}

\maketitle

\begin{abstract}
 In a recent paper, J. Lovejoy and the second author conjectured that ranks for four types of unimodal like sequences satisfy certain inequalities. In this paper, we prove these conjectures asymptotically. For this, we extend Wright's Circle Method and analyze the asymptotic behavior of certain general partial theta functions.   
\end{abstract}

\section{Introduction and statement of results}

An integer sequence is unimodal if there is a peak in the sequence. Let $u(n)$ denote the number of unimodal sequences of the form
\begin{equation} \label{useqdef}
a_1 \leq a_2 \leq \cdots \leq a_r \leq \overline{c} \geq b_1 \geq b_2 \geq \cdots \geq b_s
\end{equation} 
with weight $n = c + \sum_{j=1}^r a_j + \sum_{j=1}^s b_j$. In Ramanujan's lost notebook \cite[Entry 6.3.2]{An-Be1}, we find that  
\begin{equation} \label{entry632}
\sum_{n \geq 0} \frac{q^{n}}{( \zeta q)_n( \zeta^{-1} q)_n} = \frac{\sum_{n \geq 0} (-1)^n\zeta^{2n+1}q^{\frac{n(n+1)}{2}}}{( \zeta q)_{\infty}(\zeta^{-1} q)_{\infty}}  + (1-\zeta)\sum_{n \geq 0} (-1)^n \zeta^{3n}q^{\frac{n(3n+1)}{2}}\left(1-\zeta^2q^{2n+1}\right),
\end{equation}
where  $(a)_n=(a;q)_{n} := \prod_{k=1}^{n} (1-a q^k)$ for $n \in \mathbb{N}_{0} \cup \{\infty\}$. Since for $\zeta=1$ the left side of \eqref{entry632} becomes a generating function for $u(n)$, we can think that the coefficient of $\zeta^m q^n$, after expanding the left side, as a refinement of the number of unimodal sequences of weight $n$. This is a motivation for the definition of the unimodal rank, which is $s-r$. If we define $u(m,n)$ as the number of unimodal sequence with rank $m$, then we can see that left-hand side of \eqref{entry632} is the generating function for $u(m,n)$.  Due to the equality, though it is not clear at all combinatorially, the right-hand side of \eqref{entry632} is also the generating function for $u(m,n)$. Actually, the partial theta function on the right-hand side has played an important role in studying arithmetic property of $u(n)$ and $u(m,n)$ \cite{KL1, Wr1, Wr2}. When $\zeta=1$ the right side of \eqref{entry632} is a product of an infinite modular product and a partial theta function. Using this expression, Wright obtained asymptotics for $u(n)$ \cite{Wr1, Wr2}. Since the generating function in \eqref{entry632} is not modular, the classical Circle Method introduced by Hardy and Ramanujan does not work in this case. Wright carefully examined the asymptotic behavior of partial theta functions to employ the Circle Method.  On the other hand, Lovejoy and the second author \cite{KL1} studied the rank differences for $u(m,n)$ and congruences for certain arithmetic function involving $u(m,n)$ by analyzing the partial theta function that appeared in the generating function. In a sequel paper \cite{KL2}, Lovejoy and the second author studied the rank differences for three further different types of unimodal sequences. As the rank for $u(n)$ stems from a two variable partial theta function identity, all of these ranks are motivated from two variable partial theta function identities. These three types of unimodal ranks are denoted by $w(m,n)$, $v(m,n)$, and $\nu(m,n)$, respectively. (See Section 2 for the combinatorial definitions.)  While studying the rank differences for these unimodal ranks, Lovejoy and the second author \cite{KL2} conjectured that these ranks are weakly decreasing, i.e., for non-negative integers $m$ and $j$ with $m>j$,
\[
u(m,n) > u(j,n)
\]
holds for large enough integers $n$, and the same phenomenon occurs for the other three unimodal ranks. The main goal of this paper is confirming these conjectures asymptotically. Namely, we prove that
\begin{theorem} \label{mainthm}
For non-negative integers $m$ and $j$ with $m>j$, the inequalities
\begin{align}
u(j,n) &> u(m,n), \label{1.3}\\
w(j,n) &> w(m,n), \label{1.4}\\
v(j,n) &> v(m,n), \label{1.5}\\
\nu(j,n) &> \nu (m,n) \notag 
\end{align}
hold for all sufficiently large integers $n$. 
\end{theorem}

\begin{remarks}
\begin{enumerate}
\item[(i)] Due to the symmetry $u(m,n) = u(-m,n)$ (which also holds for the other unimodal ranks), we see that asymptotically unimodal ranks of weight $n$ are unimodal sequence with peak $u(0,n)$.
\item[(ii)] For the ranks and cranks for the ordinary partition function, the inequalities of the same type have been established by various methods \cite{BD, DM, KKS, PR}. In these cases, the generating functions are simpler, as they are (mock) modular.
\end{enumerate}
\end{remarks}

 As Wright used the asymptotic behavior of a partial theta function to obtain an asymptotic formula for $u(n)$, the asymptotic behavior of a partial theta functions also plays a crucial role in obtaining an asymptotic formula for unimodal ranks. However, as our partial theta functions are two variable functions, analyzing their asymptotic behavior is more involved. In particular, one has to show that the resulting asymptotic expansions converges. 
 
 The rest of paper is organized as follows. In Section 2, we explain what each arithmetic function $u(m,n)$, $w(m,n)$, $v(m,n)$, and $\nu(m,n)$ counts and  give their generating functions. In Section 3, we recall basic properties of certain modular forms and evaluates special kinds of integrals. In Section 4, we obtain the asymptotic behavior of a general partial theta function which is an essential part of the proof. In Section 5, by adopting Wright's Circle Method, we prove an asymptotic formula for a quite general generating function.  In Sections 6--9, we obtain asymptotic formulas for these unimodal rank functions by applying the results from Section 4. From these asymptotic formulas, Theorem \ref{mainthm} follows immediately.

\section{Unimodal generating functions}
 In this section, we introduce four types of unimodal sequences and their ranks. For the proofs, we refer the reader to \cite{KL1,KL2}.

\subsection{Unimodal sequences} \label{umnsection}

Recall that $u\left(n\right)$ denotes the number of unimodal sequences of the form \eqref{useqdef} with weight $n = c + \sum_{j=1}^r a_j + \sum_{j=1}^s b_j$. For example, $u\left(4\right)=12$, the relevant sequences being 
\begin{equation*} \label{ex1}
\begin{gathered}
\left(\overline{4}\right), \left(1,\overline{3}\right), \left(\overline{3},1\right), \left(1,\overline{2},1\right), \left(\overline{2},2\right), \left(2,\overline{2}\right), \\ \left(1,1,\overline{2}\right), \left(\overline{2},1,1\right), \left(\overline{1},1,1,1\right), \left(1,\overline{1},1,1\right), \left(1,1,\overline{1},1\right), \left(1,1,1,\overline{1}\right).
\end{gathered}
\end{equation*}
Define the rank of a unimodal sequence to be $s-r$, and assume that the empty sequence has rank $0$.  Let $u\left(m,n\right)$ be the number of unimodal sequences of weight $n$ with rank $m$.  Then the generating function for $u\left(m,n\right)$ is given by \eqref{entry632}. Note the symmetries $u\left(m,n\right) = u\left(-m,n\right)$, which follows upon exchanging the partitions $\sum_{j=1}^r a_j$ and  $\sum_{j=1}^s b_j$ in \eqref{useqdef}.

\subsection{Unimodal sequences with double peak}
 Let $w\left(n\right)$ be the number of unimodal sequences with a double peak, i.e., sequences of the form 
\begin{equation} \label{seqdefbis}
a_1 \leq a_2 \leq \cdots \leq a_r \leq \overline{c} \ \overline{c} \geq b_1 \geq b_2 \geq \cdots \geq b_s,
\end{equation} 
with weight $n = 2c + \sum_{j=1}^r a_j + \sum_{j=1}^s b_j$.  For example, $w\left(6\right)=11$, the relevant sequences being
\begin{equation*}
\begin{gathered}
\left(\overline{3},\overline{3}\right), \left(\overline{2},\overline{2},2\right), \left(2,\overline{2},\overline{2}\right), \left(\overline{2},\overline{2},1,1\right),\left(1,\overline{2},\overline{2},1\right),  \left(1,1,\overline{2},\overline{2}\right), \\
\left(\overline{1},\overline{1},1,1,1,1\right),\left(1,\overline{1},\overline{1},1,1,1\right)\left(1,1,\overline{1},\overline{1},1,1\right),\left(1,1,1,\overline{1},\overline{1},1\right),\left(1,1,1,1,\overline{1},\overline{1}\right).
\end{gathered}
\end{equation*}
Define the rank of such a unimodal sequence to be $s-r$, and assume that the empty sequence has rank $0$.  Let $w\left(m,n\right)$ denote the number of sequences counted by $w\left(n\right)$ with rank $m$. Then the generating function for $w\left(m,n\right)$ is given by (see \cite[Proposition 2.1]{KL2})
\begin{align*}
W\left(\zeta; q\right) &:=\sum_{n\geq 0} \sum_{m\in\Z} w\left(m, n\right)\zeta^m q^n 
=\sum_{n \geq 0} \frac{q^{2n}}{(\zeta q)_{n} (\zeta^{-1} q)_{n}} \\
&=\frac{\zeta^2+\left(1+\zeta^2\right)\sum_{ n\geq 1}\left(-1\right)^n \zeta^{2n} q^{\frac{n\left(n+1\right)}{2}}}{\left(\zeta q \)_{\infty} \( {\zeta^{-1} q}\right)_\infty}\\
&\quad +1  - \zeta^2+\left(1+ \zeta^2\right)\left(1-\zeta\right)\sum_{n\geq 1} \left(-1\right)^n \zeta^{3n-2} q^{\frac{n\left(3n-1\right)}{2}}\left(1+ \zeta q^n\right).
\end{align*}
Note the symmetries $w\left(m,n\right) = w\left(-m,n\right)$, which follows upon exchanging the partitions $\sum_{j=1}^r a_j$ and  $\sum_{j=1}^s b_j$ in \eqref{seqdefbis}.     
\subsection{Durfee unimodal sequences} Let $v\left(n\right)$ denote the number of unimodal sequences of the form \eqref{useqdef}, where $\sum_j {b_j}$ is a partition into parts at most $c-k$ and $k$ is the size of the Durfee square of the partition $\sum_j {a_j}$.    For example, $v\left(4\right)=10$, the relevant sequences being 
\begin{equation*} \label{ex2}
\left(\overline{4}\right), \left(1,\overline{3}\right), \left(\overline{3},1\right), \left(1,\overline{2},1\right), \left(\overline{2},2\right), \left(2,\overline{2}\right), \left(1,1,\overline{2}\right), \left(\overline{2},1,1\right), \\ \left(\overline{1},1,1,1\right),  \left(1,1,1,\overline{1}\right).
\end{equation*}
Define the rank of a sequence counted by $v\left(n\right)$ to be $s-r$, and assume that the empty sequence has rank $0$.  Let $v\left(m,n\right)$ denote the number of sequences counted by $v\left(n\right)$ with rank $m$. Then the generation function is given by (see \cite[Proposition 3.1]{KL2})
\begin{align*}
V\left(\zeta;q\right)&:=\sum_{n\geq 0}\sum_{m\in\Z} v\left(m, n\right)\zeta^m q^n =\sum_{n \geq 0} \frac{(q^{n+1})_n q^n}{(\zeta q)_n (\zeta^{-1} q)_n} \\
&=\frac{\zeta}{\left(\zeta q\)_{\infty} \(\zeta^{-1} q\right)_\infty }\sum_{n \ge 0} \zeta^{3n} q^{3n^2+2n}\left(1-\zeta q^{2n+1}\right) +\left(1-\zeta\right)\sum_{n \ge 0} \zeta^n q^{n^2+n}.
\end{align*}
Although they are not obvious from the definition, the symmetries $v\left(m,n\right) =v\left(-m,n\right)$ follow from the generating function.  

\subsection{Odd-Even unimodal sequences} Let $\nu\left(n\right)$ denote the number of unimodal sequences of the form \eqref{useqdef} where $c$ has to be odd, $\sum_j {a_j}$ is a partition without repeated even parts, and $\sum_j {b_j}$ is an overpartition into odd parts whose largest part is not $\overline{c}$.  (Recall that an overpartition is a partition in which the first occurrence of a part may be overlined.)  For example, $\nu\left(5\right)=12$, the relevant sequences being 
\begin{equation*} \label{ex3}
\begin{gathered}
\left(\overline{5}\right), \left(1,\overline{3},1\right), \left(1,1,\overline{3}\right), \left(\overline{3},1,1\right), \left(\overline{3},\overline{1},1\right), \left(1,\overline{3},\overline{1}\right), \left(2,\overline{3}\right), \\ \left(1,1,1,1,\overline{1}\right), \left(1,1,1,\overline{1},1\right), \left(1,1,\overline{1},1,1\right), \left(1,\overline{1},1,1,1\right),  \left(\overline{1},1,1,1,1\right).
\end{gathered}
\end{equation*}
Define the rank of a sequence counted by $\nu\left(n\right)$ to be the number of odd non-overlined parts in $\sum_j b_j$ minus the number of odd parts in $\sum_j a_j$, and assume that the empty sequence has rank $0$.  Let $\nu\left(m,n\right)$ denote the number of sequences counted by $\nu\left(n\right)$ with rank $m$. Then the generating function is given by (see  \cite[Proposition 4.1]{KL2})
\begin{align*}
\mathcal{V}(\zeta;q) &:=\sum_{n\geq 0} \sum_{m\in\Z} \nu\left(m, n\right)\zeta^m q^n =\sum_{n \geq 0} \frac{(-q)_{2n}q^{2n+1}}{(\zeta q ;q^2)_{n+1} (\zeta^{-1} q ;q^2)_{n+1}} \\
&=\frac{\zeta\left(-q\right)_\infty}{\left(1+\zeta\right)\left(\zeta q;q^2 \)_{\infty} \( {\zeta^{-1} q}; q^2\right)_\infty} \sum_{n\geq 0} \left(-1\right)^n \zeta^n q^{\frac{n\left(n+1\right)}{2}}-\frac{\zeta}{1+\zeta} \sum_{n \geq 0} \left(-1\right)^n \zeta^n q^{n^2+n}.
\end{align*}
Note the symmetries $\nu\left(m,n\right) = \nu\left(-m,n\right)$, which follow from exchanging the odd parts of $\sum_j a_j$ with the odd non-overlined parts of $\sum_j b_j$.

\section{Preliminaries and some integral approximations}
In this section we recall some special modular forms and their behavior under modular inversion and give some (asymptotic) integral evaluation that are required for our proofs.

\subsection{Special modular forms and Jacobi forms}
Define the usual Dedekind $\eta$ function ($q:= e^{2 \pi i \tau}$  throughout)
\[
\eta(\tau):=q^{\frac1{24}}\prod_{n\geq 1}\left(1-q^n\right)
\]
and Jacobi's theta function ($\zeta:=e^{2\pi iz}$ throughout) 
\[
\vartheta(z; \tau):=\sum_{n\in\frac12+\Z} q^{n^2} e^{2\pi in\left(z+\frac12\right)}
=-iq^{\frac18} \zeta^{-\frac12}\prod_{n\geq 1}\left(1-q^n\right)\left(1-\zeta q^{n-1}\right)\left(1-\zeta^{-1}q^n\right).
\]
We require the following transformations. 
\begin{lemma}\label{transfolemma}
We have
\begin{align*}
\eta\left(-\frac1\tau\right)&=\sqrt{-i\tau}\eta(\tau),\\
\vartheta\left(\frac{z}{\tau}; -\frac1{\tau}\right)&=-i\sqrt{-i\tau} e^{\frac{\pi iz^2}{\tau}}\vartheta(z; \tau).
\end{align*}
\end{lemma}
From Lemma \ref{transfolemma}, we directly obtain the following asymptotic  behavior which plays an  important role in the investigation of asymptotic behavior of the generating functions.  
\begin{lemma}\label{thetaas}
For $0 \leq z \leq \frac12$, we have 
\begin{equation*}
\begin{split}
\frac1{(\zeta q)_\infty\left(\zeta^{-1} q\right)_\infty}
= -i q^{\frac1{12}}\frac{e^{\frac{\pi i}{6\tau}}\zeta^{-\frac12}(1-\zeta)e^{\frac{\pi i z^2}{\tau}}}{\left(1-e^{\frac{2\pi i z}{\tau}}\right) e^{-\frac{\pi i z}{\tau}}} \left(1+O\left(e^{-2\pi (1-z)\rm{Im}\(-\frac1{\tau}\)}\right) \right).
\end{split}
\end{equation*}
\end{lemma}

The following lemma plays a key role in bounding the generating function away from the dominant pole. 
\begin{lemma}\label{Pesti}
Let $\tau = x+ i y \in \mathbb{H}$ with $y \leq |x| \leq \frac12$. Then, as $y \rightarrow 0$, 
\[
\left\lvert \frac{1}{(q)_{\infty}} \right\rvert   \ll e^{ \frac{5\pi}{72y} }.
\]
\end{lemma}
\begin{proof}
The proof follows immediately from Lemma 3.5 of \cite{BD} with $M=1$.
\end{proof}

\subsection{Integral evaluations}
In our asymptotic considerations certain integrals occur which lead to special values of Euler polynomials. To be more precise, define for $a\in\R^+$,  $\tau \in \mathbb{H}$, and $\ell\in\N_0$
\begin{align*}
 \mathcal{I}_\ell(a, \tau) &:=\int_0^a \frac{z^{2\ell+1}}{\sinh\left(\frac{\pi iz}{\tau}\right)}dz,\\
 \mathcal{K}_\ell(a, \tau) &:=\int_0^a\frac{z^{2\ell}}{\cosh\left(\frac{\pi iz}{\tau}\right)}dz.
\end{align*}
Then we have the following integral approximations.
\begin{lemma}\label{integralas}
Let $a\in\R^+$.
\begin{enumerate}
\item[(i)] 
We have, as $y\rightarrow 0$,
\[
\mathcal{I}_\ell(a, \tau) = \frac12 E_{2\ell+1}(0)\tau^{2\ell+2}+O\left(e^{-\pi a\text{\rm Im}\left(-\frac1{\tau}\right)}\right),
\]
where $E_{n}(x)$ denotes the $n$th Euler polynomial. 
\item[(ii)]
We have, as $y\rightarrow 0$,
\[
 \mathcal{K}_\ell(a, \tau) = -i E_{2\ell} \cdot\(\frac{\tau}{2}\)^{2\ell+1}+O\left(e^{-\pi a  \text{\rm Im}\left(-\frac1{\tau}\right)}\right),
\]
where $E_{n}$ is the $n$th Euler number.
\end{enumerate}
\end{lemma}

\begin{proof}
\begin{enumerate}
\item[(i)] 
 We write
\[
 \mathcal{I}_\ell(a, \tau)= \int_0^\infty \frac{z^{2\ell+1}}{\sinh\left(\frac{\pi iz}{\tau}\right)}dz-\int_a^\infty \frac{z^{2\ell+1}}{\sinh\left(\frac{\pi iz}{\tau}\right)}dz.
\]
In the first integral we make the change of variables $z\mapsto -i\tau z$ to obtain, by the Residue Theorem, that it equals
\[
(-1)^{\ell+1} \tau^{2\ell+2} \int_0^\infty \frac{z^{2\ell+1}}{\sinh(\pi z)}dz.
\]
The integral now evaluates as $\frac{(-1)^{\ell+1} E_{2\ell+1}(0)}{2}$ by Lemma 2.3 of \cite{BD}.
For the the second integral, we see that 
\begin{align*}
\int_{a}^{\infty} \frac{z^{2\ell+1}}{ \sinh \(\frac{\pi i z}{\tau} \)} dz &\ll \int_{a}^{\infty} z^{2\ell+1} e^{-\pi z \rm{Im}\(-\frac1{\tau} \) } dz\\
& \ll \(\rm{Im}\( - \frac1{\tau} \)\)^{-2\ell-2} \Gamma \( 2\ell+2, \pi a \rm{Im}\( - \frac1{\tau} \) \) \\
&\ll \rm{Im}\left(-\frac{1}{\tau}\right)^{-1} e^{-\pi a \rm{Im}\(- \frac1{\tau} \) } \ll e^{-\pi a \rm{Im}\(- \frac1{\tau} \) },
\end{align*}
where $\Gamma\(\alpha, x \):=\int_x^\infty t^{\alpha-1}e^{-t}dt$ is the incomplete gamma function and we used the fact that, as $x \rightarrow \infty$,
\[
\Gamma(k,x) \sim x^{k-1}e^{-x}.
\]
\item[(ii)] 
For the evaluation of $\mathcal{K}_\ell(a, \tau)$, 
we similarly write
\begin{equation*}
\mathcal{K}_\ell(a, \tau)= \int_{0}^{\infty} \frac{z^{2\ell}}{\cosh \(\frac{ \pi i z}{\tau}\) } dz - \int_{a}^{\infty} \frac{z^{2\ell}}{\cosh \(\frac{ \pi i z}{\tau}\) } dz.
\end{equation*}
The first integral equals
\[
i (-1)^{\ell+1} \(\frac{\tau}{\pi}\)^{2\ell+1} \int_{0}^{\infty} \frac{z^{2\ell}}{\cosh \(z \) } dz.
\]
We next find that
\begin{align*}
 \int_{0}^{\infty}  \frac{z^{2\ell}}{\cosh(z)}dz &= 2 \int_{0}^{\infty} \frac{z^{2\ell} e^{-z}}{1 + e^{-2z}}dz  = 2 \sum_{j=0}^{\infty}(-1)^{j}\int_{0}^{\infty} z^{2\ell+1} { e^{-(2j+1)z} } \frac{dz}{z} \\
& = 2 (2\ell)! \sum_{j=0}^{\infty} \frac{(-1)^{j}}{(2j+1)^{2\ell+1}}  = 2 (2\ell)!   \beta(2\ell+1) = (-1)^\ell E_{2\ell} \( \frac{\pi}{2} \)^{2\ell+1}  ,
\end{align*}
where $\beta(s): = \sum_{n=0}^{\infty} \frac{(-1)^n}{(2n+1)^s}$ is Dirichlet's $\beta$-function and we used that $\beta(2\ell+1) =  \frac{(-1)^\ell E_{2\ell} \pi^{2\ell+1}}{2^{2\ell+2}(2\ell)!}$ \cite[equation (3)]{Lima}.  
The second integral may now be bounded as before, giving the claim.
\end{enumerate}
\end{proof}

\section{Asymptotic expansion of a partial theta function}

As the generating functions we are interested in contain partial theta functions, investigating their asymptotic behavior is a crucial part of this paper.  To more  uniformly treat the occurring functions, we define the partial theta function $(d\in\Q^+, k\in\N)$ 
\begin{equation*}
F_{d,k}(z; \tau):=\sum_{n\geq 0} \zeta^{kn+d} q^{(kn+d)^2}.
\end{equation*}
The following theorem explains its asymptotic behavior near $q=1$.
\begin{theorem}\label{FasymThm}
The asymptotic expansion
\begin{equation*}
F_{d,k} (z ; \tau)= \sum_{\ell\geq 0} \frac{(2k\pi iz)^\ell}{\ell!} \left( \frac{\Gamma\left(\frac{\ell+1}{2}\right)}{2(2\pi)^{\frac{\ell+1}{2}}k^{\ell+1}}(-i\tau)^{-\frac{\ell+1}{2}}-\sum_{j=0}^N\frac{(2 k^2 \pi i )^{j}}{j!}
\frac{B_{2j+\ell+1}\left(\frac{d}{k}\right)}{2j+\ell+1}\tau^{j}\right)+O\left( |\tau|^{N+1}\right),
\end{equation*}
converges for $|z| < \frac{1}{4k}$. Here $B_n(x)$ denotes the $n$th Bernoulli polynomial. 
\end{theorem}
Before proving Theorem \ref{FasymThm}, we require an auxiliary lemma which is a slight extension of a lemma of Zagier \cite{Za}.
\begin{lemma}[Proposition 3 of \cite{Za}]\label{1.2}
Let $f : \mathbb{C} \rightarrow \mathbb{C}$ be a $C^{\infty}$ function. Furthermore, we require that $f(x)$ and all its derivatives are of rapid decay for $\rm{Re} (x) \rightarrow \infty$. Then, for $t \rightarrow \infty$ with $\rm{Re} (t) >0$ and $a > 0$, we have for any $N \in \mathbb{N}_{0}$:
\begin{align*}
\sum_{m\geq 0} f ((m+a)t) &=\frac{1}{t} \int_{0}^{\infty} f(x) dx - \sum_{n=0}^{N} \frac{f^{(n)} (0)}{n!} \frac{B_{n+1} (a)}{n+1} t^n \\
&\quad - \frac{t^N}{(N+1)!} \int_{0}^{\infty} \mathcal{B}_{N+1} \( a- \frac{x}{t} \) f^{(N+1)} (x) dx,
\end{align*}
where  $\mathcal{B}_{n} (x) := B_{n} ({ x -  \lfloor x \rfloor} )$.
\end{lemma}

We also need a following lemma, which plays an important role in showing convergence of various asymptotic expansions.

\begin{lemma}\label{lemma2}
For all $n \in \mathbb{N}_{0}$ and $\ell \in \mathbb{N}$, we have
\[
\int_{0}^{\infty} \left| f_{\ell}^{(n)} (x) \right| dx \le \frac{4^n}{2} \Gamma \left( \frac{\ell+n+1}{2} \right),
\]
where $f_{\ell} (x):=x^\ell e^{-x^2}$.
\end{lemma}

\begin{proof}
We denote 
\begin{align*}
f_{\ell}^{(n)} (x) &=:  e^{-x^2}\sum_{k=0}^{n} A_{k} (n) x^{\ell+n-2k} .
\end{align*}
Note that if $2k > \ell + n $, then $A_{k} (n) =0$. We next claim that
\begin{equation*} 
 \left|A_{k} (n) \right| \leq 2^{n-k} \binom{n}{k} (\ell+n-1)(\ell+n-3)\cdots(\ell+n-2k+1).
\end{equation*}
This bound can easily be proved by induction, using that $A_0(n)=(-2)^n$ and
\[
A_{k} (n+1) = (\ell+n-2k+2) A_{k-1} (n) -2 A_{k} (n).
\]

Therefore, 
\begin{align*}
&\int_{0}^{\infty} \left| f_{\ell}^{(n)} (x) \right| dx \le \sum_{k=0}^{n} |A_{k} (n)| \int_{0}^{\infty} x^{\ell+n-2k} e^{-x^2} dx =\frac{1}{2} \sum_{k=0}^{n} |A_k (n) | \Gamma \left( \frac{\ell+n-2k+1}{2} \right) \\
 &\le\frac{1}{2} \sum_{k=0}^{n } 2^{n-k} \binom{n}{k}2^k \frac{(\ell+n-1)}{2}\frac{(\ell+n-3)}{2}\cdots\frac{(\ell+n-2k+1)}{2} \quad\Gamma \left( \frac{\ell+n-2k+1}{2} \right)  \\
 &\le\frac{1}{2} 4^n \Gamma \( \frac{ n+\ell +1 }{2} \),
\end{align*}
where we have applied $\Gamma(x+1)=x\Gamma(x)$ $k$ times and used the Binomial Theorem.
\end{proof}

We are now ready to prove Theorem \ref{FasymThm}.

\begin{proof}[Proof of Theorem \ref{FasymThm}]
We first expand $\zeta^{kn+d}$, to obtain
\[
F_{d,k}(z; \tau)=\sum_{\ell\geq 0}\frac{(2\pi iz)^\ell}{\ell!}\sum_{n\geq 0} (kn+d)^{\ell} e^{2\pi i  (kn+d)^2 \tau}
=\sum_{\ell\geq 0}\frac{(2k\pi iz)^{\ell}}{\ell !} T^{-\ell} \sum_{n\geq 0} f_{\ell} \left(T\left(n+\frac{d}{k}\right)\right),
\]
where $T:=\sqrt{-2\pi i k^2 \tau}$. By employing Lemma \ref{1.2}, we find that the inner sum equals 
\begin{multline}\label{innersum}
 \frac{I_\ell}{T}-T^\ell \sum_{j=0}^N\frac{(-1)^j}{j!} \frac{B_{2j+\ell+1}\left(\frac{d}{k}\right)}{2j+\ell+1}T^{2j}
 -T^{\ell+2N+2} \left( \frac{(-1)^{N+1}}{(N+1)!}\frac{B_{2N+\ell+3}\left(\frac{d}{k}\right)}{2N+\ell+3} \right.\\ \left. +  \frac{1}{(2N+\ell+3)!} \int_{0}^{\infty} \mathcal{B}_{2N+\ell+3} \left( \frac{d}{k} - \frac{x}{T} \right) f_{\ell}^{(2N+\ell+3)} (x) dx \right),
\end{multline}
where 
\[
I_\ell:=\int_0^\infty f_{\ell} (x) dx = \frac{1}{2} \Gamma \(\frac{\ell+1}{2} \).
\]
Next we consider convergence of the occurring sums and show that the third and fourth summand in \eqref{innersum} contribute to the error term.
We first note that 
\begin{align}\label{bounderror}
\left| \sum_{\ell\geq 0}\frac{( 2k\pi i z)^{\ell}}{\ell !} \frac{(-1)^{j}}{j!}\frac{B_{2j+\ell+1}\left(\frac{d}{k}\right)}{2j+\ell+1} \right| \le 2 \sum_{\ell\geq 0} \frac{ (2k \pi |z|)^{\ell}}{\ell !} \frac{ (2j+\ell)!}{j! (2\pi)^{2j+\ell+1}},
\end{align} 
where we used Lehmer's bound (see Theorem 1 and equation (19) in \cite{Lehmer})
\begin{equation} \label{Lehmerbound} B_{n}(x) \leq \frac{2 n!}{(2\pi)^n}, \end{equation} 
which holds for all $x \in [0,1]$ and $n>2$. 
By the ratio test, we see that \eqref{bounderror} converges for fixed $j$ if $|z| < \frac{1}{k}$. Thus, the contributions from the second and the third term in \eqref{innersum} converge for $|z| < \frac{1}{k}$.

 We next consider the fourth term.  By Lemma \ref{lemma2} and Lehmer's bound \eqref{Lehmerbound}, we see that
\begin{align*}
&\left| \sum_{\ell\geq 0} \frac{(2k\pi i z)^{\ell}}{\ell !}  \frac{1}{(2N+\ell+3)!} \int_{0}^{\infty} \mathcal{B}_{2N+\ell+3} \left( \frac{d}{k} - \frac{x}{T} \right) f_{\ell}^{(2N+\ell+3)} (x) d x \right| \\
&\le 2 \sum_{\ell\geq 0} \frac{ (2k \pi |z|)^{\ell}}{\ell ! (2\pi)^{2N+\ell+3} } \int_{0}^{\infty} \left| f_{\ell}^{(2N+\ell+3)} (x)  \right| dx \ll  \sum_{\ell\geq 0}
\frac{(4k|z|)^\ell}{\ell!} \Gamma \left( N+\ell + 2 \right),
\end{align*}
which converges for $|z| < \frac{1}{4k}$, again using the ratio test. 

Finally, we note that
\[
\sum_{\ell \ge 0} \frac{(2k\pi i z)^{\ell}}{\ell!} \frac{I_{\ell}}{T^{\ell+1}} \le \sum_{\ell\geq 0} \frac{(2k \pi |z|)^{\ell}}{\ell!}\frac{\Gamma\( \frac{\ell+1}2\)}{2 |T|^{\ell+1}} 
\]
converges for all $z \in \mathbb{C}$ because the ratio of the coefficient 
\[
\frac{2\pi k |z| \Gamma \left( \frac{\ell+2}{2}\right)}{(\ell+1) |T| \Gamma \left( \frac{\ell+1}2\right)} =\frac{2\pi k |z| }{(\ell+1) |T|} \( \left( \frac{\ell +1 }2\right)^{\frac12} + o(1) \)
\]
tends to zero as $\ell$ goes to the infinity. Here we used that for $\alpha \in\R$
\[
\lim_{n\to\infty}\frac{\Gamma(n+\alpha)}{\Gamma(n)n^\alpha}=1.
\]
This completes the proof of Theorem \ref{FasymThm}.
\end{proof}

\section{Wright's Circle Method}

In a series of papers \cite{Wr1, Wr2}, Wright developed a generalized version of the Circle Method to obtain asymptotic formulas for the number of combinatorial functions. In this section, by adopting this method, we prove a general asymptotic formula, which can be applied to all functions of interest for this paper. 

Suppose that a function $\mathcal{F}(q)=\sum_{ n \geq 0} a(n) q^n$ has the following asymptotic expansion
\begin{equation}\label{Fqesti}
\mathcal{F}(q) = e^{\frac{\pi i}{L \tau}} \sum_{j=1}^{N} A(j) \tau^{j} + O \(|\tau|^{N+1} e^{\frac{\pi}{L} \rm{Im}\( -\frac{1}{\tau} \)}\),
\end{equation}
for some $L\in\N, N\in\N$, and $\tau = x+ iy$ with $|x|\leq y \rightarrow 0$. Moreover, we assume that there exists $\varepsilon>0$ such that for $y \le |x| \le \frac{1}{2}$
\begin{equation}\label{Fqminor}
\mathcal{F}(q) \ll  e^{\frac{\pi}{Ly}  - \varepsilon}.
\end{equation}

Under the above two assumption, by employing Wright's Circle Method, we prove the following theorem.

\begin{theorem}\label{genasym}
Suppose that $\mathcal{F}(q) = \sum_{n \ge 0} a(n) q^n$ satisfies the two assumptions \eqref{Fqesti} and \eqref{Fqminor}. Then, as $n \rightarrow \infty$,
\[
a(n) = -2\pi i \sum_{j=1}^{N} A(j)  \(\frac{i}{\sqrt{2Ln}}\)^{j+1} I_{-j-1} \( 2 \pi \sqrt{\frac{2n}{L}} \)  +O \( n^{-\frac{N+2}{2}} e^{2\pi\sqrt{\frac{2n}{L}}} \),
\]
where $I_\ell$ denotes the usual $I$-Bessel function of order $\ell$.
\end{theorem}

To determine the main contribution to $a(n)$, we need to evaluate a certain integral, namely for $s, k\in\R^+$, we define
\begin{equation*}
P_{s,k} := \frac{1}{2 \pi i} \int_{1-i}^{1+i} v^s e^{\pi \sqrt{\frac{kn}{6}} \( \frac{1}{v} + v \) } dv.
\end{equation*}
The following lemma, which is an easy generalization of a lemma of Wright \cite{Wr2}, rewrites $P_{s, k}$, up to an error term, as a Bessel function.
\begin{lemma}[Lemma 4.2 of \cite{BD}] \label{Psklemma}
As $n\to\infty$
\[
P_{s,k}  = I_{-s-1} \( \pi \sqrt{\frac{2kn}{3}} \) + O \(e^{\frac{\pi}{2} \sqrt{\frac{3kn}{2}}}\).
\]
\end{lemma}

We are now ready to prove Theorem \ref{genasym}.
\begin{proof}[Proof of Theorem \ref{genasym}]
By Cauchy's integral formula, we see that 
\begin{align*}
	a(n) &= \frac{1}{2 \pi i} \int_{\mathcal{C}} \frac{\mathcal{F}(q)}{q^{n+1}} \, dq =\int_{-\frac12}^{\frac12} \mathcal{F} \left(e^{2\pi i x- \frac{\sqrt{2}\pi}{\sqrt{Ln}}} \right)e^{\pi \sqrt{\frac{2n}{L}}-2\pi i n x} \, dx \\
	&=\int_{|x| \le \frac{1}{\sqrt{2Ln}}} \mathcal{F} \left( e^{2\pi i x- \frac{\sqrt{2}\pi}{\sqrt{Ln}}} \right)e^{\pi \sqrt{\frac{2n}{L}}-2\pi i n x} \, dx + \int_{\frac{1}{\sqrt{2Ln}} \le |x| \le \frac{1}{2} } \mathcal{F} \left( e^{2\pi i x - \frac{\sqrt{2}\pi}{\sqrt{Ln}}} \right)e^{\pi \sqrt{\frac{2n}{L}}-2\pi i n x} \, dx \\
	&=:I' + I'',
\end{align*}
where $\mathcal{C}:=\{|q|=e^{-\frac{\sqrt{2}\pi}{\sqrt{Ln}}}\}$. The integral $I'$ is the main contribution and the integral $I''$ contributes the error term as shown in the following.

We first approximate $I'$. Note that since $|x| \le y$ and ${\rm Im} (-\frac{1}{\tau}) = \frac{y}{x^2 + y^2} \le \frac{1}{y}$, the Big-O term in \eqref{Fqesti} becomes $O ( y^{N+1} e^{\frac{\pi}{Ly} })$.
 Next we evaluate, with $\tau=x+ i \frac{1}{\sqrt{2Ln}}$,
\begin{equation}\label{ieval}
\int_{|x| \le \frac{1}{\sqrt{2Ln}}} \tau^s  e^{\frac{\pi i}{L\tau}-2\pi i n \tau} \, dx =\(\frac{i}{\sqrt{2Ln}}\)^{s+1} (-2\pi i) P_{s,\frac{12}{L}}.
\end{equation}
By \eqref{Fqesti}, \eqref{ieval}, and Lemma \ref{Psklemma},  we then find that
\begin{align*}
I' = -2\pi i\sum_{j=1}^{N} A(j) \(\frac{i}{\sqrt{2Ln}}\)^{j+1} I_{-j-1} \( 2\pi \sqrt{\frac{2n}{L}} \)  +O \( n^{-\frac{N+2}{2}} e^{2\pi\sqrt{\frac{2n}{L}}} \).
\end{align*}

Moreover, by assumption \eqref{Fqminor}, it is immediate that 
\[
| I'' | \ll  n^{-\frac{N+2}{2}} e^{2\pi\sqrt{\frac{2n}{L}}},
\]
yielding the statement of the theorem. 
\end{proof}

\section{Asymptotics for $u(m,n)$}
In light of Theorem \ref{genasym}, to obtain an asymptotic formula for $u(m,n)$, it suffices to investigate  the asymptotic behavior of the generating function
\[
U_m(q):=\sum_{n\geq 0} u(m, n) q^n
\]
near and away from the dominant pole. These asymptotic behaviors are given in the following two lemmas whose proof is given at the end of this section. We start with  $q=1$. To state it, we define the constants $\alpha_{m,2k+1}$ and $\gamma_{2\ell,j}(\kappa)$ by
\begin{equation}\label{alphamk}
\zeta^{-\frac12} \left( 1-\zeta\right) \cos (2\pi m z) =: \sum_{k\geq 0} i \alpha_{m, 2k+1} z^{2k+1},
\end{equation}
 \begin{equation}\label{gamma2lj}
\gamma_{2\ell, j} (\kappa) := 
(2\kappa)^j \frac{(2\kappa \pi )^{2\ell} (-1)^{\ell}\pi^j B_{2j+2\ell+1}\left(\frac1\kappa\right)}{(2\ell ) !j!(2j+2\ell+1)} .
\end{equation} 
\begin{lemma} \label{umqmain}
For $|x| \le y$ and a positive integer $N \geq 2$, as $y \rightarrow 0$, we have
\begin{equation*}
\begin{aligned}
U_m (q) &=  e^{\frac{\pi i}{6\tau}} \sum_{\substack{k,r,s,\ell,j\geq 0 \\ 2k+r+s+2\ell+j+2 \le N}} \alpha_{m, 2k+1} \frac{(\pi i)^{r+s} (-1)^s }{12^s r! s!}
 \gamma_{2\ell,j}(4) \(\frac{i}{2} \)^j \\
 &\qquad\qquad\qquad\times E_{2k+2r+2\ell+1} (0) \tau^{2k+r+s+2\ell+j+2} + O\(  |\tau|^{N+1} e^{\frac{\pi}{6}{\rm Im}\(-\frac{1}{\tau} \)} \).
\end{aligned}
\end{equation*}
\end{lemma}
The next lemma gives the behavior of $U_m(q)$ away from $q=1$.
\begin{lemma}\label{umqerror}
For $y \le |x| \le \frac{1}{2}$ and some $\varepsilon>0$, we have
\begin{equation*}
U_{m} (q) \ll e^{\frac{\pi}{6y} - \varepsilon }.
\end{equation*}
\end{lemma}

From the above two lemmas, the asymptotic formula for $u(m,n)$ is immediate.

\begin{theorem}\label{umnasymthm}
For  $m \in \N_{0}$ and an integer $N \geq 2$, we have, as $n \rightarrow \infty$,
\begin{align*}
u(m,n) &=\sum_{\substack{k,r,s,\ell,j\geq 0 \\ 2k+r+s+2\ell+j+2 \le N}}  2^{1-j}  (-1)^{k+r+s+j+\ell+1} \alpha_{m, 2k+1} \frac{\pi^{r+s+1}}{12^s r! s!}  E_{2k+2r+2\ell+1} (0) \gamma_{2\ell, j} (4) \\
&\qquad\qquad\qquad\qquad\times X_{2k+r+s+2\ell+j+3} (n)
+ O \( n^{-\frac{N+2}2} e^{2\pi \sqrt{\frac{n}{3}}} \),
\end{align*}
where $X_{k} (n) :=  (2\sqrt{3n})^{-k}I_{-k} (2\pi \sqrt{n/3})$.
\end{theorem}

\begin{proof}[Proof of Theorem \ref{umnasymthm}]
 Using Lemmas \ref{umqmain} and \ref{umqerror}, we find that $U_m (q)$ satisfies the two assumptions \eqref{Fqesti} and \eqref{Fqminor} required for Theorem \ref{genasym}. By applying Theorem \ref{genasym} with $L=6$ and
\[
A(M):=\sum_{\substack{k,r,s,\ell,j \geq 0 \\ 2k+r+s+2\ell+j+2 = M}} \alpha_{m, 2k+1} \frac{(\pi i)^{r+s} (-1)^s }{12^s r! s!} \gamma_{2\ell,j}(4) \(\frac{i}{2} \)^j  E_{2k+2r+2\ell+1} (0),
\]
we deduce the asymptotic formula for $u(m,n)$ as claimed in Theorem \ref{umnasymthm}.
\end{proof}

In particular, choosing $N=4$ in Theorem \ref{umnasymthm}, yields by a direct calculation
\begin{corollary} \label{umnasym}
For  $m\in \N_0$, we have,  as $n \rightarrow \infty$,
\begin{align*}
u(m,n) &= \frac{\pi^2}{2} X_{3} (n) + \frac{\pi^3}{3} X_{4} (n) +\frac{\pi^4}{72} \( 59-36 m^2 \) X_5 (n) + O \( n^{-3} e^{2\pi \sqrt{\frac{n}{3}}} \).
\end{align*}
\end{corollary}
Corollary \ref{umnasym} now immediately gives the inequalities for $u(m, n)$.
\begin{proof}[Proof of \eqref{1.3}]
Corollay \ref{umnasym} yields that
\[
u(j,n) - u(m,n) \sim \frac{\pi^4}{2}  \frac{m^2 - j^2}{\left(2\sqrt{3n}\right)^5}  I_{-5} \( 2 \pi \sqrt{\frac{n}{3}} \),
\]
which directly implies the claim since $I_\ell(x)>0$ for $x\in\R^+$.
\end{proof}

Now we turn to proving Lemma \ref{umqmain}.
\begin{proof}[Proof of Lemma \ref{umqmain}]
We start with noting that Cauchy's integral formula and the symmetry $u(-m,n) =u(m,n)$ imply,
\begin{equation}\label{symmint}
U_m(q)=2\int_0^{\frac12} U(\zeta; q)\cos(2\pi mz)dz.
\end{equation}
 Using \eqref{entry632}, we decompose the generating function as
\begin{align*}
U(\zeta; q) =G_{u,1} (\zeta ;q) + G_{u,2} (\zeta ; q),
\end{align*}
where
\begin{align*}
G_{u, 1}(\zeta; q)&:=\frac{\sum_{n\geq 0}(-1)^n \zeta^{2n+1} q^{\frac{n(n+1)}{2}}}{(\zeta q)_\infty \left(\zeta^{-1}q\right)_\infty},\\
G_{u, 2}(\zeta; q)&:=(1-\zeta) \sum_{n\geq 0} (-1)^n \zeta^{3n} q^{\frac{n(3n+1)}{2}}\left(1-\zeta^2 q^{2n+1}\right).
\end{align*}
 We first approximate the partial theta function occuring in $G_{u, 1}$. By splitting into even and odds, we obtain
\begin{equation}\label{partialtheta}
\sum_{n\geq 0} (-1)^n \zeta^{2n+1} q^{\frac{n(n+1)}{2}} 
=q^{-\frac18}\left(F_{1,4} \left(z; \frac{\tau}{8}\right)-F_{3,4} \left(z; \frac{\tau}{8}\right)\right).
\end{equation}
By Theorem \ref{FasymThm}, we find that for $|z| < 1/16$,  \eqref{partialtheta}  has the asymptotic expansion
\[
q^{-\frac18} \sum_{\ell\geq 0} \frac{(8\pi iz)^\ell}{\ell!}\sum_{j=0}^N\frac{(4\pi i)^j}{j!}
\frac{\left(B_{2j+\ell+1}\left(\frac34\right)-B_{2j+\ell+1}\left(\frac14\right)\right)}{2j+\ell+1} \tau^{j} +O\left( |\tau|^{N+1}\right).
\]
Thus, by employing Lemma \ref{thetaas},  we have for $z\in(0, 1/16)$ the asymptotic expansion 
\begin{equation}
\label{asG1}
\begin{aligned}
G_{u,1} \left( \zeta; q \right) &= 
\frac{2iq^{-\frac{1}{24}} e^{\frac{\pi i}{6\tau}} \zeta^{-\frac12} \left( 1-\zeta\right) e^{\frac{\pi i z^2}{\tau}}}{\left( 1-e^{\frac{2\pi i z}{\tau}}\right)e^{-\frac{\pi i z}{\tau}}}  
\sum_{\ell\geq 0} \frac{(8\pi i z)^{2\ell}}{(2\ell)!} \sum_{j=0}^N \frac{(4\pi i)^{j}}{j!} \frac{ B_{2j+2\ell+1}\left(\frac14\right)}{2j+2\ell+1} \tau^j \\&\quad
+O \left( |\tau|^{N+1} e^{ { \frac{83 \pi }{768}} {\rm Im} \(-\frac{1}{\tau}\)}\right) ,
\end{aligned}
\end{equation}
where we used that $B_{k} (x)=(-1)^kB_k(1-x)$ and that $z-z^2 \le 15/256$ for $z \in (0,1/16)$.

 Moreover, for $1/16 \le z \le 1/2$, we can bound 
\begin{equation}\label{g1error}
 G_{u,1} (\zeta ; q)  \ll  e^{{ \frac{83 \pi }{768}}\text{Im}\left(-\frac1\tau\right)}  \sum_{n\geq 0}  e^{-\frac{\pi n(n+1)y}2}  \ll |\tau|^{-\frac12} e^{{ \frac{83 \pi }{768}}\rm{Im}\(-\frac1{\tau}\)},
\end{equation}
where we used that $y\gg|\tau|$ and Lemma \ref{thetaas} to estimate the contribution from the infinite product. 

For $G_{u,2}$ we bound directly for $ 0 \leq z \leq 1/2$ 
\begin{equation}\label{g2bound}
\left\lvert (1-\zeta)\sum_{n\geq 0} (-1)^n \zeta^{3n} q^{\frac{n(3n+1)}{2}}\left(1-\zeta^2 q^{2n+1}\right)\right\rvert
\ll\sum_{n\geq 0}  e^{-2\pi  n^2 y}  \ll |\tau|^{-\frac12}.
\end{equation}

Therefore, decomposing the integral in \eqref{symmint} as 
\begin{align*}
U_m(q) &= 2 \int_0^{\frac{1}{16} }  G_{u,1} (\zeta; q) \cos(2\pi mz)dz \\ 
& \qquad + 2\int_{\frac{1}{16}}^{\frac12} G_{u,1} (\zeta; q) \cos(2\pi mz)dz + 2\int_0^{\frac{1}{2}} G_{u,2} (\zeta; q) \cos(2\pi mz)dz \\
&=: M_{u} (q) +E_{u,1}(q) + E_{u,2} (q), 
\end{align*}
we observe, by \eqref{g1error} and \eqref{g2bound}, that
\begin{equation}
\label{errorbound}
E_{u,1} (q) + E_{u,2}  (q) \ll |\tau|^{-\frac12} e^{\frac{83\pi}{768} \rm{Im} \( -\frac1{\tau} \) } .
\end{equation}
On the other hand, by \eqref{asG1}, Lemma \ref{integralas} (i), and by expanding $e^{\frac{\pi iz^2}{\tau}}$,  we deduce that $M_u(q)$ equals 
\begin{equation}\label{expandmain}
\begin{aligned}
 &i q^{-\frac{1}{24}} e^{\frac{\pi i}{6\tau}}  \sum_{\substack{k,r,\ell,j\geq 0 \\ 2k+r+2\ell+j+2 \le N}} i \alpha_{m, 2k+1} \frac{( \pi i )^r}{r!} \left(-2 i^j \right) 2^{-j} \gamma_{2\ell,j} (4) \tau^{j-r} \int_0^{\frac{1}{16}} \frac{z^{2k+1+2r+2\ell}}{\sinh \left( \frac{\pi  i z}{\tau}\right)}dz \\&\qquad\qquad\qquad\qquad\qquad\qquad\qquad\qquad\qquad\qquad\qquad\qquad\qquad\qquad\qquad + O\( |\tau|^{N+1} e^{\frac{\pi}{6} {\rm Im}\(- \frac{1}{\tau} \) } \)  \\
&=q^{-\frac{1}{24}} e^{\frac{\pi i}{6 \tau}} \sum_{\substack{k,r,\ell,j\geq 0 \\ 2k+r+2\ell+j+2 \le N}} \alpha_{m, 2k+1} \frac{(\pi i)^{r}}{r!} \gamma_{2\ell,j}(4) \(\frac{i}{2}\)^j  E_{2k+2r+2\ell+1} (0) \tau^{2k+r+2\ell+j+2} \\
&\qquad\qquad\qquad\qquad\qquad\qquad\qquad\qquad\qquad\qquad\qquad\qquad\qquad\qquad\qquad + O\( |\tau|^{N+1} e^{\frac{\pi}{6}{\rm Im}\(-\frac{1}{\tau} \)}  \),
\end{aligned}
\end{equation}
where $\alpha_{m, 2k+1}$ and $\gamma_{2\ell,j} (4)$ are defined by \eqref{alphamk} and \eqref{gamma2lj}, respectively. 
Therefore, combining \eqref{errorbound} and \eqref{expandmain} and expanding $q^{-\frac1{24}}$, gives the claimed asymptotic expansion. 
\end{proof}
We now turn to the proof of Lemma \ref{umqerror}.
\begin{proof}[Proof of Lemma \ref{umqerror}]
Recall that \cite[Theorem 2.1]{BCCL}
\begin{equation}\label{crankpartial}
\frac1{(\zeta q)_\infty\left(\zeta^{-1} q\right)_\infty} =  \frac{1-\zeta}{(q)_{\infty}^2}\sum_{n \in \mathbb{Z}} \frac{(-1)^n q^{\frac{n(n+1)}{2}}}{1-\zeta q^n}.
\end{equation}
Approximating
\begin{align}\label{Lerchbound}
\left\lvert (1-\zeta) \sum_{n \in \mathbb{Z}} \frac{(-1)^n q^{\frac{n(n+1)}{2}}}{1-\zeta q^n} \right\rvert &\ll 1+\sum_{n \ge 1} \frac{|q|^{\frac{n\left(n+1\right)}{2}}}{1-|q|^n} \ll 1+\frac{1}{1-|q|}\sum_{n\geq1}e^{-2\pi n^2 y} \ll y^{-\frac32},\\ \notag
\left\lvert \sum_{n\geq 0} (-1)^n \zeta^{3n} q^{\frac{n(3n+1)}{2}} \right\rvert
&\ll\sum_{n\geq 0} e^{-2\pi n^2 y}  \ll y^{-\frac12},
\end{align}
we obtain
\[
U(\zeta; q)\ll y^{-2}\frac{1}{|(q)_\infty|^2}+y^{-\frac12}.
\]
In summary, by combining the above bounds with Lemma \ref{Pesti}, in the region $y \le |x| \le \frac{1}{2}$, we have
\[
U_{m} (q) = 2\int_0^{\frac12}U(\zeta; q)\cos(2\pi mz)dz 
\ll y^{-2} \left\lvert\frac{1}{(q)_{\infty}^{2}} \right\rvert  { + y^{-\frac12} } 
\ll e^{\left( \frac{\pi}{6y} - \varepsilon \right)},
\]
as desired.   
\end{proof}

\section{Asymptotics for $w(m,n)$}
The following two lemmas give the asymptotic behavior of the generating function $W_{m} (q) := \sum_{n} w(m,n) q^n$ near and away from $q=1$. Firstly, we have near $q=1$.

\begin{lemma}\label{wmnasym1}
For $|x|\leq y$ and an integer $N \geq 2$, we have, as $y \rightarrow 0$,
\begin{multline*}
W_{m} (q) =  e^{\frac{\pi i}{6\tau}} \left( \frac12 \sum_{\substack{ k,r,t \geq 0 \\ 2k+r+t+2 \le N}  }  \alpha_{m, 2k+1} \frac{( \pi i )^{r+t}}{6^t r!t!} {E_{2k+2r+1} (0)} \tau^{2k+r+t+2} \right. \\
 \qquad\qquad+2\sum_{\substack{k, r, j,\ell,s,t \ge 0 \\ j+r+t+2k+2\ell+2s+2 \le N} } \alpha_{m, 2k+1} \frac{(-1)^t ( \pi i )^{r+t}}{12^t r! t!}   \( \frac{i}{2} \)^j \gamma_{2\ell,j} (4)  \frac{ (-1)^{s} (2\pi)^{2s}}{(2s)!}  \\
 \left. \vphantom{\sum_{\substack{ k,r,t \geq 0 \\ 2k+r+t+2 \le N}  }} \quad\quad\times E_{2r+2k+2\ell+2s+1}(0) \tau^{j+r+t+2k+2\ell+2s+2} \right) + O\( |\tau|^{N+1} e^{\frac{\pi}{6} {\rm Im}\(- \frac{1}{\tau} \) } \).
\end{multline*}
\end{lemma}
\noindent
Since the proof of the above lemma is similar to that of Lemma \ref{umqmain}, we omit it here. 

By using Lemma \ref{Pesti}, \eqref{crankpartial}, and proving as before
\[
W(\zeta; q)\ll y^{-2}\frac{1}{|(q)_\infty|^2}+y^{-\frac12},
\]
we deduce the following asymptotic behavior away from $q=1$.
\begin{lemma}\label{wmnasym2}
For $y \le |x| \le \frac{1}{2}$, we have, for some $\varepsilon>0$, 
\[
W_{m} (q) \ll e^{\frac{\pi}{6y} - \varepsilon}.
\]
\end{lemma}

From Lemmas \ref{wmnasym1} and \ref{wmnasym2}, we find that ${W}_{m} (q)$ satisfies the two assumptions required for Theorem \ref{genasym}. Thus, by applying this theorem, we deduce that 
\begin{theorem} \label{wmasymthm} 
For $m\in\N_0$ and an integer $N \geq 2$, we have, as $n \rightarrow \infty$,
\begin{align*}
w(m,n) &= \sum_{\substack{ k,r,t \geq 0 \\ 2k+r+t+2 \le N}  }  (-1)^{r+t+k+1} \alpha_{m, 2k+1}  \frac{\pi^{r+t+1}}{6^t r! t!} E_{2k+2r+1} (0) X_{2k+r+t+3} (n) \\ 
 &\quad+  4 \sum_{\substack{k, r, j,\ell,s,t \ge 0 \\ j+r+2k+2\ell+2s+t+2 \le N} } (-1)^{r+t+k+j+\ell+s+1} \alpha_{m, 2k+1}  \frac{\pi^{r+t+2s+1}2^{2s-j}}{12^t r! t! (2s)!} \gamma_{2\ell,j}(4)  \\
 &\quad\quad\quad \times E_{2r+2k+2\ell+2s+1}(0) X_{2k+r+t+2\ell+2s+j+3} (n) 
 + O \( n^{-\frac{N+2}2} e^{2\pi \sqrt{\frac{n}{3}}} \).
\end{align*}
\end{theorem}

In particular, $N=5$ yields, by a lengthy but straightforward calculation, the following asymptotic main terms. 
\begin{corollary} \label{wmnasym}
For a fixed non-negative integer $m$, we have,  as $n \rightarrow \infty$,
\begin{align*}
w(m,n) = \frac{\pi^3}{3} X_4 (n) + \frac{55\pi^4 }{24} X_5 (n) + \frac{\pi^5 (1841- 108 m^2 )}{324} X_6 (n) +   O \( n^{-\frac{7}{2}} e^{\pi \sqrt{\frac{4n}{3}}} \).
\end{align*}
\end{corollary}

Inequality \eqref{1.4} in Theorem \ref{mainthm} is now immediate from the above corollary.

\section{Asymptotics for $v(m,n)$}
The following two lemmas, whose proof we omit, describe the asymptotic behavior of $V_m(q):=\sum_{n \ge 0} v(m,n) q^n$ near and away from the dominant pole.
Near $q=1$, we have
\begin{lemma}\label{vmnasym1}
For $|x| \leq y$ and an integer $N \ge 2$, we have, as $y \rightarrow 0$, 
\begin{multline*}
V_{m} (q) = e^{\frac{\pi i}{6 \tau}} \sum_{\substack{k,r,s,\ell,j\geq 0 \\ 2k+r+s+2\ell+j+2  \le N }} \alpha_{m, 2k+1} \frac{(-1)^s (\pi i)^{r+s}}{2^s r! s!} \gamma_{2\ell,j}(3) i^j  E_{2k+2r+2\ell+1} (0) \tau^{2k+r+s+2\ell+j+2} \\
\qquad + O\( |\tau|^{N+1} e^{\frac{\pi}{6}{\rm Im}\(-\frac{1}{\tau} \)}  \).
\end{multline*}
\end{lemma}

By using Lemma \ref{Pesti}, \eqref{crankpartial}, and proving as before that
\[
V(\zeta; q)\ll y^{-2}\frac{1}{|(q)_\infty|^2}+y^{-\frac12},
\]
we deduce the following asymptotic behavior away from $q=1$.
\begin{lemma}\label{vmnasym2}
For $y \le |x| \le \frac{1}{2}$, we have, for some $\varepsilon>0$,
\[
V_{m} (q) \ll e^{\frac{\pi}{6y} - \varepsilon }.
\]
\end{lemma}

From Lemmas \ref{vmnasym1} and \ref{vmnasym2}, we find that $V_{m} (q)$ satisfies the two assumptions in Theorem \ref{genasym}. Thus, again using Theorem \ref{genasym}, we deduce the following result.
\begin{theorem} \label{vmnasymthm} 
For a fixed non-negative integer $m$ and a positive integer $N \geq 2$, we have as $n \rightarrow \infty$,
\begin{align*}
v(m,n) &=\sum_{\substack{k,r,s,\ell,j\geq 0 \\ 2k+r+s+2\ell+j+2 \le N}} 2 (-1)^{k+r+s+j+\ell+1} \alpha_{m, 2k+1} \frac{\pi^{r+s+1}}{2^s r! s!}  E_{2k+2r+2\ell+1} (0) \gamma_{2\ell, j} (3)\\
 & \qquad\qquad\qquad\qquad \times X_{2k+r+s+2\ell+j+3} (n)
 + O \( n^{-\frac{N+2}2} e^{\pi \sqrt{\frac{4n}{3}}} \).
\end{align*}
\end{theorem}

In particular, $N=4$ yields the following asymptotic main terms.

\begin{corollary} \label{vmnasym}
For a fixed non-negative integer $m$, we have,  as $n \rightarrow \infty$,
\begin{align*}
v(m,n) = \frac{\pi^2}{3} X_3 (n)  + \frac{4\pi^3}{27} X_4 (n) + \frac{\pi^4 \left(101- 72m^2\right) }{216}  X_5 (n) + O \( n^{-3} e^{\pi \sqrt{\frac{4n}{3}}} \).
\end{align*}
\end{corollary}

Inequality \eqref{1.5} in Theorem \ref{mainthm} is now immediate from the above corollary.

\section{Asymptotics for $\nu(m,n)$}
As the generating function of $\nu(m,n)$ contains a quotient of two Jacobi theta functions, investigating its asymptotic behavior requires more work, but still fits into the general method developed in Sections 3 and 4. The following two lemmas describe the asymptotic behavior of $\mathcal{V}_m (q):= \sum_{n \ge 0 } \nu(m,n)q^n$. We start with the asymptotic behavior near $q=1$. For this let $\gamma_{2\ell,j} (4)$ be given as in \eqref{gamma2lj}, and $\beta_{m,2k}$ is the constant defined by
$$
\frac{\cos (2\pi m z)}{\zeta^{-\frac{1}{2}} + \zeta^{\frac{1}{2}}} =: \sum_{k\geq 0} \beta_{m, 2k} z^{2k}.
$$
The proofs of the following two lemmas are given at the end of this section.
\begin{lemma}\label{numqasym1}
For $|x|<y$ and a positive integer $N$, as $y \rightarrow 0$,
\begin{align*}
\mathcal{V}_{m} (q) & = { \sqrt{2}e^{\frac{\pi i}{8\tau}} }  \sum_{\substack{k,r,\ell,j,s \geq 0\\ j+r+2k+2\ell+1 \le N}  } (-1)^s \beta_{m, 2k} \frac{( \pi i )^{r+s} }{ 2^{r+s { + j+2\ell}} r! s!} i^{j+1}  \gamma_{2\ell,j}(4) E_{2k+2r + 2\ell} \tau^{j+r+s+2k+2\ell+1}   \\
&\qquad\qquad\qquad\qquad\qquad\qquad  + O\( |\tau|^{N+1} e^{\frac{\pi}{8} {\rm Im}\(- \frac{1}{\tau} \) } \).
\end{align*}  
\end{lemma}
Away from $q=1$, we have the following behavior.
\begin{lemma}\label{numqasym2}
For $y \le |x| \le \frac{1}{2}$, we have, for some $\varepsilon>0$,
\[
\mathcal{V}_{m} (q) \ll e^{\frac{\pi}{8y} - \varepsilon  }.
\]
\end{lemma}

 Lemmas \ref{numqasym1} and \ref{numqasym2} enable us to apply Theorem \ref{genasym} to obtain the following asymptotic formula for $\nu(m,n)$.

\begin{theorem}
For $m \in \N_{0}$ and $N \in \N$, we have 
\begin{align*}
\nu (m,n) &= 2^{\frac32} \sum_{\substack{k,r,\ell,j,s \geq 0\\ j+r+s+2k+2\ell+1 \le N}  } \beta_{m, 2k} \frac{  \pi^{r+s+1} }{ 2^{ r+s+j {+ 2\ell} } r! s!} \gamma_{2\ell,j}(4) (-1)^{j+r+k + \ell +1}  E_{2k+2r+2\ell} Y_{j+r+s+2k+2\ell+2} (n) \\
&\qquad \qquad\qquad\qquad\qquad\qquad
+ O  \( n^{-\frac{N+2}{2}} e^{\pi \sqrt{n}} \),
\end{align*}
where $Y_k (n) := (4\sqrt{n})^{-k} I_{-k} ( \pi \sqrt{n} )$.
\end{theorem}

By expanding first three non-zero terms, we find the following asymptotic main terms.

\begin{corollary}
We have
\begin{align*}
{\nu(m,n) =\frac{\pi}{2\sqrt{2}} Y_2 (n) + \frac{5\pi^2}{8\sqrt{2}} Y_3(n) +
 \frac{\pi^3 (77 - 64 m^2 ) }{64 \sqrt{2}} Y_4 (n) + O \( n^{-\frac 52} e^{\pi \sqrt{n} } \).}
\end{align*}
\end{corollary}

Now we prove Lemma \ref{numqasym1}.

\begin{proof}[Proof of Lemma \ref{numqasym1}]
As before, we write
\[
\mathcal{V} \left( \zeta; q \right) = G_{\nu,1} (\zeta; q) + G_{\nu,2} (\zeta; q), 
\]
where 
\begin{align*}
G_{\nu, 1}(\zeta; q)&:=\frac{\zeta(-q)_\infty}{(1+\zeta)\left(\zeta q ;q^2\)_{\infty} \( \zeta^{-1} q; q^2\right)_\infty}\sum_{n\geq 0} (-1)^n\zeta^n q^{\frac{n(n+1)}{2}},\\
G_{\nu, 2}(\zeta; q)&:=-\frac{\zeta}{1+\zeta}\sum_{n\geq 0} (-1)^n \zeta^n q^{n^2+n}.
\end{align*}
By splitting the partial theta function into even and odd terms, we find that 
\[
G_{\nu,1} (\zeta; q) = \frac{\zeta^{\frac12} q^{-\frac{1}{8}} \( \zeta q^2 ;q^2 \)_{\infty} \( \zeta^{-1} q^2 ; q^2\)_{\infty}\( q^2 ; q^2 \)_{\infty}}{(1+\zeta) \left( \zeta q \)_{\infty} \(\zeta^{-1} q \)_{\infty} \( q \right)_\infty} \(F_{\frac12,2}\left(z; \frac{\tau}{2} \right)-F_{\frac32,2}\left(z; \frac{\tau}{2}  \right) \),
\]
where we used that
\begin{equation*}
\frac{(-q)_\infty}{\left(\zeta q ;q^2\)_{\infty} \( \zeta^{-1} q ;q^2\)_{\infty} }  = \frac{ (\zeta q^2 ;q^2)_{\infty} ( \zeta^{-1} q^2; q^2)_{\infty} ( q^2 ; q^2)_{\infty}}{\left(\zeta q \)_{\infty}  (\zeta^{-1} q)_{\infty} \( q\right)_\infty   }.
\end{equation*}

By employing Lemma \ref{transfolemma}, we first approximate the contribution coming from the infinite product, namely, we have, 
\begin{equation}\label{thetaap}
\frac{\zeta^{\frac12} q^{-\frac{1}{8}} \( \zeta q^2 ;q^2 \)_{\infty} \( \zeta^{-1} q^2 ; q^2\)_{\infty}\( q^2 ; q^2 \)_{\infty}}{(1+\zeta) \left( \zeta q \)_{\infty} \(\zeta^{-1} q \)_{\infty} \( q \right)_\infty} 
= \frac{\zeta^{\frac12} q^{-\frac{1}{4}} e^{\frac{\pi i}{ 8 \tau}} e^{\frac{\pi i z^2}{2 \tau}} }{\sqrt{2}(1+\zeta) e^{-\frac{\pi i z}{2\tau}}\left(1+e^{\frac{\pi i z}{\tau}}\right)} \( 1+ O\left(e^{-\pi(1-z)\rm{Im}\(-\frac1{\tau}\)}\right)\).
\end{equation}

Combining this with Theorem \ref{FasymThm}, we obtain that for $0 \le z \le \frac{1}{4}$
\begin{align*}
G_{\nu,1} \left( \zeta; q \right) &= -\frac{ { \sqrt{2} }\zeta^{\frac12} q^{-\frac{1}{4}} e^{\frac{\pi i}{ 8 \tau}} e^{\frac{\pi i z^2}{2 \tau}} }{(1+\zeta) e^{-\frac{\pi i z}{2\tau}}\left(1+e^{\frac{\pi i z}{\tau}}\right)} 
\sum_{\ell\geq 0} \frac{(4\pi i z)^{2\ell}}{(2\ell)!} \sum_{j=0}^N \frac{(4\pi i)^{j}}{j!} \frac{ B_{2j+2\ell+1}\left(\frac14\right)}{2j+2\ell+1} \tau^{j} \\
&\qquad\qquad\qquad\qquad\qquad\qquad\qquad\qquad\qquad\qquad\qquad\qquad\quad+O \left( |\tau|^{N+1} e^{ { \frac{\pi }{32}} {\rm Im} \(-\frac{1}{\tau}\) }\). 
\end{align*}
In the remaining range $\frac{1}{4} \le z \le \frac12$, we bound
\begin{equation*}
G_{\nu,1} (\zeta ;q) \ll |\tau |^{-\frac{1}{2}} e^{ { \frac{\pi }{32}} {\rm Im} \(-\frac{1}{\tau} \)}.
\end{equation*}
Finally, for $0 < z< \frac12$, we have
\begin{equation*}
G_{\nu,2} (\zeta; q) \ll |\tau |^{-\frac{1}{2}}.
\end{equation*}
%
Proceeding as before, but using the second formula in Lemma \ref{integralas} finishes the proof.
\end{proof}

We next bound the generating function away from the dominant pole.

\begin{proof}[Proof of Lemma \ref{numqasym2}]
Exactly as before, one can show that all contributions other then those from the infinite product have at most polynomial growth in $1/y$, and thus it suffices to show that for $y \le |x| \le \frac{1}{2}$
\[
\frac{(-q)_{\infty} }{\left(\zeta q ;q^2\)_{\infty} \( \zeta^{-1}q ; q^2\right)_{\infty}}  \ll e^{\frac{\pi}{8y} - \varepsilon  },
\]
for some $\varepsilon>0$.
From (\ref{crankpartial}) we obtain that
$$
\frac{(-q)_\infty}{\left( \zeta q; q^2 \right)_\infty
\left( \zeta^{-1} q; q^2 \right)_\infty} = 
\frac{1}{(q)_\infty 
\left( \zeta q^2; q^2 \right)_\infty} 
\sum_{n\in\Z}
\frac{(-1)^n q^{n(n+1)}}
{1-\zeta q^{2n+1}}.
$$
As in (\ref{Lerchbound}), we obtain
$$
\left| \sum_{n\in\Z}
\frac{(-1)^n q^{n(n+1)}}{1-\zeta q^{2n+1}} \right|
\ll y^{-\frac32}.
$$
To bound the remaining product, we write
$$
\log \left( \frac{1}{ (q)_\infty \left(  q^2; q^2 \right)_\infty}\right)=
\sum_{n\geq 1} \frac{q^n}{n}
\left( \frac{1}{1-q^n} + 
\frac{1}{1-q^{2n}} \right),
$$
which implies that
$$
\log \left( \left|
\frac{1}{(q)_\infty \left(  q^2; q^2 \right)_\infty}\right|\right)\leq
\sum_{n\geq 1} \frac{|q|^n}{n}
\left( \frac{1}{1-|q|^n} + 
\frac{1}{1-|q|^{2n}} \right) - |q| \left( \frac{1}{1-|q|} - 
\frac{1}{|1-q|} \right).
$$
The sum on $n$ now equals
$$
\log \left( \frac{1}
{\left(  |q|; |q| \right)_\infty \left(  |q|^2; |q|^2 \right)_\infty } \right)= 
\frac{\pi}{8y} + O \left( \log(y) \right).
$$
Moreover, as in the proof of Lemma 3.5 in \cite{BD}, we may bound
\begin{align*}
1-|q| &= 2\pi y + O\left( y^2 \right),\\
|1-q| &\geq 2 \sqrt{2} \pi y.
\end{align*}
This easily yields the claim.
\end{proof}

\end{document}